\documentclass[11pt]{amsart}

\usepackage{fullpage}
\usepackage{amsthm}
\usepackage{amssymb}
\usepackage{amsmath}
\usepackage{graphicx}
\usepackage{tikz}
\usepackage{mathrsfs}
\usepackage{float}
\usepackage{makecell}
\usepackage{amscd}
\usepackage{multirow}
\usepackage{xspace}
\usepackage{verbatim}
\usepackage{hhline}
\usepackage{cite}
\usepackage{booktabs}

\newcommand{\R}{\mathbb{R}}

\newcommand{\E}{\mathbb{E}}

\newcommand{\Var}{\text{Var}}

\newcommand{\spcap}{\mathcal{C}}
\newcommand{\B}{\mathcal{B}}

\renewcommand{\epsilon}{\varepsilon}

\DeclareMathOperator{\dist}{dist}
\DeclareMathOperator{\Bin}{Bin}

\newcommand\polymake{\texttt{polymake}\xspace}

\usepackage{relsize}
\newtheorem{theorem}{Theorem}
\newtheorem{lemma}[theorem]{Lemma}

\newtheorem{corollary}[theorem]{Corollary}
\newtheorem{claim}[theorem]{Claim}
\theoremstyle{definition}

\newtheorem{definition}[theorem]{Definition}
\newtheorem{question}[theorem]{Question}
\newtheorem{remark}[theorem]{Remark}
\title{Doubly random polytopes}
\author{Andrew Newman}
\thanks{Technische Universit\"at Berlin (former), Carnegie Mellon University (current)}
\address{Technische Universit\"at Berlin, Chair of Discrete Mathematics/Geometry,  Strasse des 17. Juni 136, 10623 Berlin, Germany (affiliation at time of writing); Carnegie Mellon University, 5000 Forbes Ave, Pittsburgh, PA 15213 (current affiliation)}

\email{anewman@andrew.cmu.edu}
\date{\today}
\begin{document}
\maketitle
\begin{abstract}
A two-step model for generating random polytopes is considered. For parameters $d$, $m$, and $p$, the first step is to generate a simple polytope $P$ whose facets are given by $m$ uniform random hyperplanes tangent to the unit sphere in $\R^d$, and the second step is to sample each vertex of $P$ independently with probability $p$ and let $Q$ be the convex hull of the sampled vertices. We establish results on how well $Q$ approximates the unit sphere in terms of $m$ and $p$ as well as asymptotics on the combinatorial complexity of $Q$ for certain regimes of $p$.\\

\noindent
Keywords: random polytopes, random graphs, probabilistic method
\end{abstract}
\section{Introduction}

A standard family of models for random polytopes comes from the convex hull of points chosen randomly from the interior or the boundary of some fixed convex body $K$. There is quite extensive literature on this model from a variety of perspectives for example \cite{BMT, Borgwardt, BeneathBeyond, Raynaud, Barany89, SW2003, SW20032, BFV, GiftWrapping, ReitznerBoundary}. On the other hand there are random 0/1 polytopes studied by \cite{Ziegler2000, Kaibel2004, Furedi1986, BaranyPor2001, KKS1995}. Here we examine a model first posed by Michael Joswig that is in some sense a combination of these two perspectives. 

For parameters $d$ and $m$ positive integers, with $m$ sufficiently larger than $d$, and $p \in [0, 1]$ we will sample a random polytope $Q$ from our model $P_2(d, m, p)$ via the following series of steps.  First generate a polytope $P$ as the convex hull of $m$ random points on the boundary of the unit sphere in $\R^d$, then take its dual $P^{\circ}$. Next generate $Q$ from $P^{\circ}$ by taking the convex hull a random collection of the vertices of $P^{\circ}$ obtained by including each vertex of $P^{\circ}$ in the collection independently with probability $p$. This is naturally a composition of two models for random polytope and it will be helpful to have notation for each one. The first model we will denote by $P_1(d, m)$ which generates a polytope via the convex hull of $m$ random points on the unit sphere in $\R^d$. The second model we will denote $\B(P, p)$ for $P$ a polytope and $p \in [0, 1]$ which will generate a random polytope via the convex hull of a binomial random sample of the vertices of $P$ including each vertex of $P$ in the sample independently with probability $p$. Thus $Q \sim P_2(d, m, p)$ is $Q \sim \B(P^{\circ}, p)$ where $P \sim P_1(d, m)$.

The decision to start with a random inscribed polytope and then take its polar dual, rather than to directly start with a random circumscribed polytope is just because this approach makes some of the arguments easier. One could sample $P^{\circ}$ directly by taking $m$ hyperplanes tangent to the unit sphere and $P^{\circ}$ as the intersection of the half-space containing the origin for each hyperplane. In this way $P_2(d, m, p)$ is related to models of random polytopes generated by random half-spaces as studied for example in \cite{KellyTolle, SchmidtMattheiss, HugReitznerSchneider}. 

An important result on $P_1(d, m)$ that we use extensively when studying $P_2(d, m, p)$ is the following result of Buchta, M\"{u}ller, and Tichy. Here and throughout we use $f_i$ to denote the $i$th entry of the $f$-vector of a polytope, i.e. its number of $i$-dimensional faces.

\begin{theorem}[\hspace{1sp}\cite{BMT}]\label{BMTtheorem}
Fix $d \geq 2$, for $P \sim P_1(d, m)$,
\[\lim_{m \rightarrow \infty} \frac{\E(f_{d - 1}(P))}{m}\]
is an explicit constant $F(d)$ defined by 
\[F(d) := \frac{2}{d} \gamma_{(d - 1)^2} \gamma_{d - 1}^{-(d - 1)}\]
where $\left\{\gamma_k\right\}_{k = 0}^{\infty}$ given by the recurrence $\gamma_0 = \frac{1}{2}$ and $\gamma_{k + 1} = \frac{1}{2 \pi (k + 1) \gamma_k}$. 
\end{theorem}
Clearly $F(2) = 1$ and $F(3) = 2$; the first few nontrivial values are 
\begin{eqnarray*}
F(4) &=& \frac{24 \pi^2}{35} \approx 6.77, \\
F(5) &=& \frac{286}{9} \approx 31.78, \text{ and}\\
F(6) &=& \frac{1296000 \pi^4}{676039} \approx 186.74.
\end{eqnarray*}
Now $F(d)$ grows very quickly, from the recurrence of Theorem \ref{BMTtheorem} one can show that it grows like $\exp(\Theta(d \log d))$, but nonetheless in fixed dimension $d$, $P \sim P_1(d, m)$ has complexity which is $O(m)$. Recall that the \emph{complexity} of a polytope is the sum of the entries of its $f$-vector. McMullen's upper bound theorem tell us that the complexity of a $d$-dimensional polytope on $n$ vertices is $O(n^{\lfloor d/2 \rfloor})$. Moreover \emph{cyclic polytopes} realize this upper bound on the complexity, see for example \cite{Ziegler1995} for more background about extremal complexity of polytopes. 

The low complexity of polytopes in $P \sim P_1(d, m)$ is a motivation to study the two-step model. From a computational perspective, one may hope to compute the double description of a polytope in a reasonable amount of time only if the polytope has low complexity. For $P \sim P_1(d, m)$, for example, a special case of the main result of Borgwardt in \cite{BeneathBeyond} is that for $d$ fixed computing the double description of $P \sim P_1(d, m)$ has expected algorithmic complexity of order $O(m^2)$. By proving low complexity results on $Q \sim P_2(d, m , p)$ we can establish classes of low complexity polytopes that are \emph{neither simple nor simplicial}, unlike $P \sim P_1(d, m)$ which is a simplicial polytope with probability 1. This allows for ways to generate polytopes with new combinatorial types, either theoretical though the study of $Q \sim P_2(d, m, p)$, or, when low complexity holds, explicit computation of the $f$-vector of instances of $Q \sim P_2(d, m, p)$.

A particular application for this doubly random model comes from work of Joswig, Kaluba, and Ruff \cite{JKR}. The authors of \cite{JKR} have applied random techniques for polytopes in the context of machine learning, approximating the shape of data via randomized constructions and using the resulting polytope as a classifier.  In this case it is critical to have a rich class of polytopes of low complexity. For such polytopes it is reasonable to expect that computing a double description is feasible in practice. Toward this goal we introduce the following definition suggested by Joswig.
\begin{definition}
  Let $P_1,P_2,\dots$ be a family of $d$-dimensional polytopes in $\R^d$ such that $\lim_{i\to\infty} f_0(P_i)=+\infty$.
  The family is \emph{slender} if there are constants, $c,c'>0$, only depending on $d$, such that \[c\cdot f_0(P_i)\leq f_{d-1}(P_i) \leq c'\cdot f_0(P_i)\] for all $i$.
\end{definition}
Now $P \sim P_1(d, m)$, $m \rightarrow \infty$ form a slender family of polytopes with $c' = F(d) + \epsilon$ and $c = F(d) - \epsilon$ for any $\epsilon > 0$ and $m$ large enough. By establishing low complexity results on $P_2(d, m, p)$ one may probabilistically establish families of slender polytopes which will in general be neither simple nor simplicial.

We show here that for $p$ sufficiently close to 1, $Q \sim P_2(d, m, p)$, $m \rightarrow \infty$ form a slender family of polytopes with explicit $c$ and $c'$. How close $p$ must be to 1 for the result to hold will depend on $d$ but not on $m$. We also establish results about how close $Q \sim P_2(d, m, p)$ is to the unit sphere in $\R^d$; this has direct applications to \cite{JKR} which we outline in the appendix.

\section{Results}
We are primarily interested in number of facets for $Q \sim P_2(d, m, p)$ in the situation where $d$ and $p$ are fixed and $m$ tends to infinity. The main result in this direction is the following upper bound on the expected number of facets.

\begin{theorem}\label{upperbound}
For each $d \geq 2$ there exists a constant $C_d$ so that for $p >1 -  1/(8(d - 1))$ the expected number of facets of $Q \sim P_2(d, m, p)$ is asymptotically at most
\[\left(1 + F(d)p^dq \left(1 + \frac{C_d q}{1 - 8(d - 1)q} \right) \right) m \]
where $q = 1- p$.
\end{theorem}
The restriction that $p$ be close to 1 seems to be an artifact of the proof rather than a natural phenomenon of $P_2(d, m, p)$, we include some remarks about other regimes of $p$ at the end of the article. 

We also provide a general lower bound on the expected number of facets together with a concentration inequality that is useful for constructing slender families of polytopes with some control over $c$ and $c'$. The lower bound holds for all values of $p$.
\begin{theorem}\label{lowerbound}
For $Q \sim P_2(d, m, p)$ with $d$ fixed and $p \in [0, 1]$, for any $\delta > 0$ with high probability the number of facets of $Q$ is at least
\[(1 - \delta)(p^{d } + F(d) p^d q)m \]
where $q = 1 - p$
\end{theorem}
Here and throughout we use ``with high probability" to mean that a property holds with probability converging to 1 as $m$ tends to infinity. From the bounds in Theorems \ref{upperbound} and \ref{lowerbound} we see that the closer $p$ is to 1 the more accurate the estimate on the number of facets of $Q \sim P_2(d, m, p)$.

Lastly we consider the question of how well $Q \sim P_2(d, m, p)$ approximates the unit sphere via the Hausdorff distance. Observe that $P \sim P_1(d, m)$ is a polytope inscribed in the unit sphere and $P^{\circ}$ is circumscribed around the unit sphere. However, once we sample $Q$ as the convex hull of a sample of vertices of $P^{\circ}$, we no longer have a guarantee that $Q$ is circumscribed around the unit sphere anymore. Thus to understand how close $Q$ is to the sphere we have to establish how close the vertices of $Q$, which are outside the unit sphere, are to the sphere and how far we ``cut into" the sphere when we remove the vertices of $P^{\circ}$ that do not belong to the sampled vertices for $Q$. By considering both of these questions, we establish the following for $d$ fixed, $m$ tending to infinity, and $p$ either fixed or depending on $m$.

\begin{theorem}\label{Hausdorffdistance}
For $d \geq 2$ fixed and $pm^{1/(d - 1)}$ tending to infinity as $m$ tends to infinity, with high probability the Hausdorff distance between $Q \sim P_2(d, m, p)$ and the unit sphere is at most
\[O \left(\frac{ \log^2(pm) \log^{2/(d - 1)}(m)}{p^2m^{2/(d - 1)}} + \frac{\log^{2/(d-1)}(pm^d)}{m^{2/(d - 1)}} \right). \]
\end{theorem}
The proof of Theorem \ref{Hausdorffdistance} proceeds by showing a lower bound on the largest ball contained in $Q$ and an upper bound on the smallest ball containing $Q$. As a corollary to this method of proof we therefore also obtain upper and lower estimates on the volume of $Q \sim P_2(d, m, p)$.

\section{Complexity in the dense regime}\label{sec:complexity}
In this section we describe the expected number of facets of $P \sim P_2(d, m, p)$ when $p$ is close to 1. The upper and lower bounds on the expected number of faces will be nontrivial for $p > 1 - \delta_d$ where $\delta_d$ is a constant depending only on $d$. Moreover the expected number of faces in this regime will be bounded between $c(d, p) m$ and $C(d, p) m$ where $C(d, p)$ and $c(d, p)$ approach the same value as $p$ tends to one. However, the results do not require that $p$ tends to one with $m$ in order to be able to give useful upper and lower bounds. 

The bounds that are established apply to a more general class of random polytope. In some sense the bounds depend more heavily on the second step of the two-step model than on the first. The upper bound for the dense regime follows from an general upper bound that we prove for simple polytopes $P$ in the model $\B(P, p)$ that holds provided that all subsets of $d + 2$ vertices that do not lie on a facet of $P$ do not lie in the same hyperplane. For simple polytopes, this condition is referred to a \emph{(primal) nondegeneracy} from the point of view of linear programming. Moreover, it is a reasonable assumption for us; $P \sim P_1(d, m)$ will have all of its vertices in general position with probability 1, so $P^{\circ}$ will be nondegenerate. 

To study the facets of $Q \sim \B(P, p)$ for $P$ a simple polytope we observe that the set of facets naturally partitions into \emph{old facets}, those facets that arise as the restriction of the binomial model to a facet of $P$, and \emph{new facets}, those facets that are not contained in any single facet of $P$. Another way to define these notions is by examining the number of vertices of $P$ ``above" a facet of $Q$. To make this precise, one may assume that $P$ contains the origin in its interior and for a facet $\sigma$ of $Q$, we let $\mathcal{C}(\sigma)$ denote the set of vertices of $P$ on the hyperplane determined by $\sigma$ or in the half-space determined by $\sigma$ not containing the origin. More succinctly $\mathcal{C}(\sigma)$ is the set of vertices of $P$ on or above the hyperplane determined by $\sigma$. If $Q$ still contains the origin in its interior, which holds with high probability in the regime of $m$ and $p$ we are interested in, then the old facets of $Q$ are all those facets $\sigma$ with $\mathcal{C}(\sigma) = \sigma$ and the rest are the new facets. A first observation toward enumerating the new facets is the following claim:

\begin{claim}\label{uniquecaps}
If $\sigma$ is a new facet of $Q \sim \B(P, p)$, then the set of vertices on or above above the hyperplane determined by $\sigma$ in $P$ induce a connected subgraph of the 1-skeleton of $P$, whose vertices we will denote by $\spcap(\sigma)$. 
\end{claim}
\begin{proof}
Let $\sigma$ be a new facet of $Q \sim \B(P, p)$. At any vertex $v$ in $\spcap(\sigma)$ the value of the linear functional given by the normal vector to the hyperplane determined by $\sigma$ is nonnegative. Moreover, the simplex method starting at $v$ will find a path from $v$ to an optimal vertex $w$ with respect to this functional and in doing so the path will always stay inside $\spcap(\sigma)$.
\end{proof}

Ultimately we want to upper bound the number of new facets of $Q \sim P_2(d, m, p)$ by enumeration of choices of $\spcap(\sigma)$. By Claim \ref{uniquecaps} we see that every $\spcap(\sigma)$ is connected so it becomes important to have an enumeration formula for the number of connected subgraphs of the 1-skeleton of $P^{\circ}$ for $P \sim P_1(d, m)$. Indeed we have the following which holds for any $d$-regular graph. A result of this type also appears in a different context in \cite{NP}.
\begin{claim}\label{connectedgraphs}
If $G$ is a $d$-regular graph on $n$ vertices then for each $t \geq 2$, the number of connected induced subgraphs of $G$ on $t$ vertices is at most \[4^{t - 2} d (d - 1)^{t - 2} n.\]
\end{claim}
\begin{proof}
If $G$ is $d$-regular and $v$ is a vertex of $G$, then there exists a connected induced subgraph on $t$ vertices containing $v$ if and only if there is an injective graph homomorphism from some rooted tree $T$ on $t$ vertices into $G$ mapping the root to $v$. Now if we let $\mathcal{T}(t)$ denote the number of rooted trees on $t$ vertices then the number of connected induced subgraphs of $G$ containing $v$ is at most
\[\mathcal{T}(t) d (d - 1)^{t - 2}. \]
Indeed once we have picked the tree to map, the choice for mapping the root is fixed, we have $d$ choices for the first neighbor of the root that we map to $G$ and then at most $(d - 1)$ choices for each vertex after that. We only now have to bound $\mathcal{T}(t)$.

The question of the number of rooted trees has been extensively studied and $|\mathcal{T}(t)|$ is known to grow exponentially in $\mathcal{T}$; see for example \cite{Otter} where the author shows that for $t$ large enough there is a constant $\alpha \approx 2.95$ so that $|\mathcal{T}(t)| \approx \alpha^t$. In the interest of giving a self-contained proof here and having a bound for all $t$, we show that for $t \geq 2$, $|\mathcal{T}(t)| \leq 4^{t - 2}$.

Observe that every rooted tree of size $t$ may be encoded uniquely by a sequence of $(t - 1)$ nonnegative integers. Given a lexicographic labeling of $T \in \mathcal{T}(t)$, that is the root is labeled with 0, its $k$ neighbors are labeled $1, 2 ..., k$, then the neighbors of vertex 1 labeled consecutively starting at $k + 1$ and so on, we can encode $T$ with a $(t - 1)$-sequence of nonnegative integers $(a_1, ..., a_{t - 1})$ where $a_i$ denotes the number of children of vertex $i - 1$.  Moreover we have $1 + a_1 + \cdots + a_{t - 1} = t$ and since $t \geq 2$, $a_1 > 0$. Thus we bound the number of such sequences. 

By the requirement that the sum of all the entries of our sequence is $t - 1$, we may obtain all sequences via a procedure that takes $t - 1$ vertices arranged in a line, vertical bars at the beginning and the end, and inserts $t - 2$ additional bars anywhere in between two vertices or a vertex and a bar, except that the sequence cannot begin with two consecutive vertical bars. The sequence may be read off by counting the number of points between consecutive bars. Note that some bars may have no vertices between them. The number of outcomes to such a procedure is 
\[\left( \! \binom{t - 1}{t - 2} \! \right) = \binom{t - 1 + t - 2 - 1}{t - 2} < 2^{2t - 4} = 4^{t - 2}\]
where $\left( \! \displaystyle \binom{N}{r} \! \right)$ is multiset coefficient notation, the number of ways to choose a multiset of size $r$ from a set of $N$ elements.
\end{proof}
\begin{theorem}\label{upperbounddeterministic}
For each $d \geq 2$ there exists a constant $C_d$ so that if $P$ is a nondegenerate simple $d$-polytope with $m$ facets and $n$ vertices, then for $p > 1 - 1/(8(d - 1))$ the expected number of facets of $Q \sim \B(P, p)$ is at most 
\[m + nqp^d \left(1 + \frac{C_d q}{1 - 8(d-1)q} \right), \]
where $q = 1 - p$.
\end{theorem}
\begin{proof}
For each $t \geq (d + 1)$, we may upper bound the expected number of new hyperplanes $\sigma$ so that $\mathcal{C}(\sigma)$ has exactly $t$ vertices. For $t = (d + 1)$, the expected number of $\sigma$ with $\mathcal{C}(\sigma) = d + 1$ is $nqp^d$. We choose a vertex, and the vertices of its link must define the new facets with that vertex above it. The selected vertex must be excluded from the sample while the vertices of its link must be included. For $t \geq (d + 2)$ fixed, an overestimate for the expected number of hyperplanes $\sigma$ so that $\mathcal{C}(\sigma)$ has exactly $t$ vertices is given by first counting the number of choices of $t$ vertices of the graph of $P$ so that the graph induced by those vertices is connected. Next, since the convex hull of $\mathcal{C}(\sigma)$ is a polytope on $t$ vertices with $\sigma$ as a face, from each choice of $t$ vertices inducing a connected graph we have at most $2^t$ choices for a facet $\sigma$. Once we have chosen the set of $t$ vertices and the facet, which is necessarily a simplicial facet by the nondegeneracy assumption on $P$, we have a probability of at most $q^{t - d}p^d$ that the chosen facet is a new facet of $Q$ with the remaining $t - d$ vertices above it. 

By Claim \ref{connectedgraphs} we have that the expected number of new facets is at most
\begin{eqnarray*}
&& n \left( qp^d + \sum_{t = d + 2}^{\infty} 4^{t - 2} d(d - 1)^{t - 2} 2^t q^{t - d}p^d  \right) \\
 &&\leq n \left(qp^d + \frac{dp^d}{16q^d(d - 1)^2} \sum_{t = d + 2}^{\infty} \left(8(d - 1)q\right)^t \right) \\
 &&\leq n \left(qp^d + 4(8)^d(d - 1)^dq^2 dp^d \sum_{t = 0}^{\infty} \left(8(d - 1)q\right)^t \right) \\
 &&\leq n \left(qp^d + \frac{4(8)^d(d - 1)^dq^2 dp^d }{1 - 8(d - 1)q} \right). \\
\end{eqnarray*}
This gives the upper bound on the number of new facets and trivially $m$ is an upper bound on the number of old facets
\end{proof}

Now as $m \rightarrow \infty$ the number of vertices in $P \sim P_1(d, m)$ tends to $F(d)m$. Therefore we have that the expected number of facets of $Q \sim P_2(d, m, p)$ satisfies the bound in Theorem \ref{upperbound} as a corollary.

We now turn our attention to a lower bound on the expectation. We first give a general lower bound valid for all $p$, and concentration inequality on the expected number of new facets. 
\begin{theorem}\label{shallowcuts}
If $P$ is a simple $d$-polytope with $m$ facets and at least $n$ vertices then for every $p \in [0, 1]$, the expected number of facets of $Q \sim \B(P, p)$ is at least
\[p^{d}m + nqp^d,\]
where $q = 1 - p$. And moreover for any $\epsilon > 0$ the probability that the number of facets is at least $(1 - \epsilon)(p^{d }m + nqp^d)$ is at least 
\[1 - e^{-2 (\epsilon qp^d/d)^2 n} - e^{-(m p^{d} \epsilon^2)/2}.\]
\end{theorem}
\begin{proof}
Since $P$ is a simple polytope, the link of every vertex in the 1-skeleton is a collection of $d$ vertices. For a fixed vertex $v$, and $\{w_1, ..., w_d\}$ its neighbors, the facet given by the convex hull of $\{w_1, ..., w_d\}$ will be included in $Q$ if every vertex $w_1, ... , w_d$ is included in the sample and $v$ is excluded from the sample. Call such a facet a shallow cut. The expected number of shallow cuts in $Q$ is clearly at least $nqp^d$, and this is a lower bound for the number of new facets in $Q$. 

For the concentration of measure statement we apply McDiarmid's inequality from \cite{McDiarmid}
\begin{theorem}[\hspace{1sp}\cite{McDiarmid}]
Let $X_1, ..., X_n$ be random $\{0, 1\}$-valued random variables. If $f$ is a real-valued function on $X_1, ..., X_n$ so that for every $i$ there exists $c_i$ so that for all $(X_1, .., X_n)$
\[|f(X_1, ..., X_i, ..., X_n) - f(X_1, ..., 1 - X_i, ..., X_n)| \leq c_i.\]
Then for every $t > 0$
\[\Pr(\E(f(X_1, ..., X_n)) - f(X_1, ..., X_n) \geq t) \leq \exp\left(-2t^2/\sum_{i = 1}^n c_i^2\right) .\]
\end{theorem}
In our case the $X_i$'s are the indicator random variables for the vertices of $P$ sampled for $Q$, and $f(X_1, ... ,X_n)$ counts the number of shallow cuts. Thus we have that for any $i$, 
\[|f(X_1, ..., X_i, ... X_n)  - f(X_1, ...., 1 - X_i, ..., X_n)| \leq d.\]
Indeed swapping a vertex into or out of the sample only affects the shallow cuts involving adjacent vertices. Thus for $\epsilon > 0$ fixed the probability that the number of facets is at most $(1 - \epsilon) nqp^d$ is at most the probability that the number of shallow cuts is at most $(1 - \epsilon) n q p^d$ which is bounded above by
\[\Pr(\E(f(X_1, ..., X_n)) - f(X_1, ..., X_n)  \geq \epsilon n qp^d). \]
By McDiarmid's inequality this is at most
\[\exp\left(-2 \epsilon^2 q^2 p^{2d} n / d^2\right), \]
so this is an upper bound on the probability that there are fewer than $(1 - \epsilon)p^dqn$ new facets.

Next we find a lower bound on the number of old facets. If $P$ is a simple polytope with $m$ facets and we sample $Q \sim \B(P, p)$, then each facet of $P$ has at least $d$ vertices and so it contributes an old facet to $Q$ with probability at least $p^{d}$. So the expected number of old facets is at least $p^{d}m$. Moreover, if we select $d$ vertices from each facet of $P$ ahead of taking the sample, then we have that the number of old facets in $Q$ stochastically dominates the binomial random variable $\Bin(m, p^{d})$, so we have that for any $\epsilon > 0$ the probability that $Q$ has fewer than $(1 - \epsilon) m p^{d}$ facets is at most
\[\exp( - (mp^{d} \epsilon^2)/2). \]
The concentration inequality on the total number of facets now follows.
\end{proof}

To prove Theorem \ref{lowerbound}, we also need a concentration inequality on the number of facets of $P \sim P_1(d, m)$. There are already concentration inequalities on $f$-vector entries for random polytopes sampled as the convex hull of points in the \emph{interior} of a fixed convex body \cite{Reitzner2005}, however the concentration inequality we prove here for points sampled from the boundary of a ball is apparently new.  Our concentration inequality is obtained via the \emph{Efron--Stein jackknife inequality} first described in \cite{EfronStein}. The Efron--Stein inequality is applied by Reitzner \cite{Reitzner2003} to study concentration of volumes for polytopes obtained by taking the convex hull of points from the interior or from the boundary of some fixed convex body $K$. In the context of random polytopes one has that if $K$ is a convex body and one builds a sequence of polytopes $K_1, K_2, ...., K_m$ where $K_1$ is a single point sampled according to some distribution $\mu$ on $K$ and $K_i$ is sampled by taking the convex hull of $K_{i  - 1}$ and a new point sampled by $\mu$, then for a functional $f$ of the random polytope process the Efron--Stein jackknife inequality is that
\[\Var f(K_m) \leq (m + 1) \E((f(K_{m+1}) - f(K_m))^2).\]
Thus one can bound the variance by understanding how the polytope changes when a single point is added. A result of Reitzner that is used in \cite{Reitzner2003} to study the volume of a random polytope is well suited here to prove the concentration inequality we will need.
\begin{theorem}[Special case of Theorem 10 of \cite{Reitzner2003}]\label{ReitznerLemma}
Let $X_1, ..., X_m, X$ be points chosen uniformly on the unit sphere in $\R^d$, and let $S_m(X)$ be the random variable counting the number of facets of the convex hull of $X_1, ..., X_m$ which are no longer facets in the convex hull of $X_1, .., X_m, X$.
 Then there exists a constants $C_d$ and $c_d$ so that
\[\lim_{m \rightarrow \infty} \E(S_m(X)) = C_d,\]
and 
\[\lim_{m \rightarrow \infty} \E(S_m(X)^2) = c_d. \]
\end{theorem}
With the result we will prove the following about the concentration of the number of facets of $P \sim P_1(d, m)$.
\begin{theorem}\label{vertexconcentration}
For $\epsilon > 0$ fixed the probability that the number of facets $f_{d - 1}(P)$ for $P \sim P_1(d, m)$ satisfies $|f_{d - 1}(P) - F(d)m| \geq \epsilon F(d) m$ is $O \left(1/m\right)$.
\end{theorem}
\begin{proof}
Consider a sequence of polytopes $K_1, ..., K_m, ...$ where $K_1$ is a single point selected uniformly at random from the unit sphere in $\R^d$ and $K_i$ is obtained as the convex hull of $K_{i - 1}$ and a new point selected uniformly at random from the unit sphere in $\R^d$. By the Efron--Stein inequality we have that
\[\Var f_{d - 1}(K_m) \leq (m + 1) \E((f_{d - 1}(K_{m+1}) - f_{d - 1}(K_m))^2).\]
Now $|f_{d - 1}(K_{m + 1}) - f_{d - 1}(K_m)|$ is at most $S_m(X)(d + 1)$. This can be seen because if $K_m$ is already constructed then the number of facets erased by sampling a new point is exactly $S_m(X)$ by definition, however sampling a new point will also create new facets. But the number of new facets is controlled by $S_m(X)$ as well. Every new facet will come from the cone of a face of dimension $(d - 2)$ of some facet of $S_m(X)$ and cone point given by the new vertex. As every polytope $K_1, ..., K_m, ...$ will be simplicial with probability $1$ there are at most $dS_m(X)$ new facets added. Thus 
\[\E((f_{d - 1}(K_{m+1}) - f_{d - 1}(K_m))^2) \leq \E(((d + 1)S_m(X))^2), \]
and this tends to $(d + 1)^2 c_d$ by Theorem \ref{ReitznerLemma}. It follows that in the limit 
\[ \Var f_{d - 1}(K_m) \leq (d + 1)^2(m + 1) c_d .\]
Now as $\E(f_{d - 1}(K_m)) \rightarrow F(d) m$ we have by Chebyshev's inequality that 
\[ \Pr(|f_{d - 1}(K_m) - F(d) m| \geq \epsilon F(d)m ) \leq \frac{\Var(f_{d - 1}(K_m))}{\epsilon^2 (\E(f_{d - 1}(K_m)))^2} \rightarrow  \frac{(d + 1)^2(m + 1) c_d}{\epsilon^2 F(d)^2 m^2} = O \left(\frac{1}{m}\right). \]
\end{proof}

\begin{proof}[Proof of Theorem \ref{lowerbound}]
Let $\delta > 0$ be given and suppose $P \sim P_1(d, m)$. By Theorem \ref{vertexconcentration}, the probability that $P$ has fewer than $(1 - \delta)F(d)m$ facets as $m$ tends to infinity is $O(1/m)$. If $P$ has at least $(1 - \delta)F(d)m$ facets then $P^{\circ}$ has at least $(1 - \delta)F(d)m$ vertices and $m$ facets so we are in the situation to apply Theorem \ref{shallowcuts} with $n = (1 - \delta)F(d)m$ and in this case with probability $\exp(-\Omega(m))$ $Q \sim \B(P^{\circ}, p)$ has fewer than 
\[\left(p^{d} + (1 - \delta) qp^d F(d) \right)m \]
facets. So for $Q \sim P_2(d, m, p)$ the probability that $Q$ has fewer than $\left(p^{d } + (1 - \delta)qp^d F(d)\right) m$ facets is $O(1/m) + \exp(-\Omega(m)) = o(1)$. 
\end{proof}

\begin{remark}
The concentration statements for the lower bounds are convenient to have for the purpose of constructing slender families of polytopes. If we take fixed $q < 1/(8(d - 1))$, then Theorem \ref{upperbound} gives us a constant $C = C(d, p)$ with $p = 1 - q$ so that asymptotically in $m$ the expected number of facets of $Q \sim P_2(d, m, p)$ is at most $C(d, p) m$, so by Markov's inequality there is a positive probability that such a random polytope has at most $(1 + \delta)C(d, p)m$ facets for any fixed $\delta > 0$. On the other hand Theorem \ref{lowerbound} gives us a constant $c = c(d, p)$ \emph{and} a concentration inequality so that we can say that with high probability $Q \sim P_2(d, m, p)$ has at least $(1 - \delta)c(d, p)m$ facets. The number of vertices of $Q \sim P_2(d, m, p)$ is concentrated around $p F(d) m$ by Theorem \ref{vertexconcentration}. For $\delta > 0$ and $m$ sufficiently large then $Q \sim P_2(d, m, p)$ satisfies the following:
\begin{enumerate}
\item with high probability $Q$ has between $(1 - \delta) pF(d) m$ and $(1 + \delta) pF(d)$ vertices.
\item with high probability $Q$ has at least $(1 - \delta) c m$ facets, and
\item with positive probability $Q$ has at most $(1 + \delta) C m$ facets.
\end{enumerate}
So with positive probability, for $\delta > 0$ and $m$ large enough, $Q \sim P_2(d, m, p)$ satisfies
\[\frac{(1 - \delta) c f_0(Q)}{(1 + \delta) p F(d)} < f_{d - 1}(Q) < \frac{(1 + \delta) C f_0(Q)}{(1 - \delta) p F(d)}.\]

\end{remark}

\section{Approximation of the sphere}

Here we estimate the Hausdorff distance between $Q \sim P_2(d, m, p)$ and the unit sphere. We will bound the Hausdorff distance to the unit sphere $S^{d - 1}$ by showing that the boundary of $Q$ lives outside a sphere centered at the origin of radius $1 - \epsilon$ and inside a sphere centered at the origin of radius $1 + \epsilon$, where we will have $\epsilon$ depend on $m$ and $p$ and keep the assumption that $d$ is fixed. 

The first step toward bounding the Hausdorff distance for $Q$ is to bound how far away the vertices of $Q$ are from the origin. 
\begin{lemma}\label{HausdorffOutside}
For $\epsilon > 0$ bounded above by a sufficiently small constant and $Q \sim P_2(d, m, p)$ with probability at least 
\[1 - p\binom{m}{d} \left(1 - \frac{\sqrt{\epsilon}^{d - 1}v_{d - 1}}{2dv_d} \right)^{m - d}\]
 every vertex of $Q$ is at distance at most $(1 + \epsilon)$ from the origin where $v_k$ denotes the Lebesgue measure of the unit ball in $\R^k$.
\end{lemma}
\begin{proof}
Let $P \sim P_1(d, m)$ and $Q \sim \B(P^{\circ}, p)$. Let $X_Q$ be the random variable counting the number of vertices of $Q$ at distance at least $(1 + \epsilon)$ from the origin. Since the sample of vertices of $P^{\circ}$ selected to define $Q$ are chosen uniformly and independently of the distance of those points from the origin, $\E(X_Q) = p\E(X_{P^{\circ}})$ where $X_{P^{\circ}}$ is the number of vertices of $P^{\circ}$ at distance at least $1 + \epsilon$ from the origin. Each vertex of $P^{\circ}$ at distance at least $(1 + \epsilon)$ from the origin is dual to a hyperplane of $P$ at distance at most $\frac{1}{1 + \epsilon}$ from the origin which induces a facet. The hyperplane defining such a facet therefore cuts off a cap of height at least $1 - \frac{1}{1 + \epsilon} = \frac{\epsilon}{1 + \epsilon}$. Moreover this cap can contain no vertices of $P$.

The probability a uniform random point of the unit sphere in $\R^d$ lives on some fixed cap of height $\frac{\epsilon}{1 + \epsilon}$ is at least
\[\frac{\sqrt{\epsilon}^{d - 1} v_{d - 1}}{2dv_d}.\]
Thus the expected number of facets of $P$ whose defining hyperplane is at distance at most $\frac{1}{1 + \epsilon}$ from the origin is by linearity of expectation at most
\[\binom{m}{d} \left(1 - \frac{\sqrt{\epsilon}^{d - 1}v_{d - 1}}{2dv_d} \right)^{m - d}.\]
And from this estimate and Markov's inequality the bound on the probability follows.
\end{proof}

The second piece of the argument is to establish how close to the origin the facets of $Q$ may be. The argument for this part proceeds in essentially two steps, the first is to show that $P^{\circ}$ for $P \sim P_1(d, m)$ has that all its facets have metric diameter $O((\log m/m)^{1/(d - 1)})$ with high probability. Then we show that conditioned on this metric diameter bound for $P^{\circ}$, with high probability $Q \sim \B(P^{\circ}, p)$ has that every new facet has $O(\log m)$ vertices of $P^{\circ}$ above it. When this holds, the distance from any point of the unit sphere above a hyperplane determining a new facet to that facet must be at most the length of a longest path from a point of $P^{\circ}$ above the hyperplane to a vertex of the new facet. Since every edge has Euclidean length $O((\log m/m)^{1/(d - 1)})$ and any such path has $O(\log m)$ edges, the distance from the origin to the new facet may not be smaller than $1 - O([(\log m) (\log m/m)^{1/(d - 1)}]^2)$. This argument is made precise by the following series of claims.

In the first claim we deal with the diameter of a facet. For polytope diameter has a few different meanings, here we mean the metric diameter; i.e. For a polytope $P$ we use ``diameter" to mean the supremum distance under the Euclidean metric between any two points of $P$.

\begin{claim}\label{triangle}
If $P$ is a polytope circumscribed around the unit sphere in $\R^d$ with every vertex of $P$ at distance at most $(1 + \epsilon)$ from the origin for $\epsilon > 0$, then the diameter of every facet of $P$ is at most $\sqrt{12 \epsilon + 6 \epsilon^2}$.
\end{claim}
\begin{proof}
Let $u$ and $v$ be points in some facet $\sigma$ of $P$ and let $\Delta$ denote the triangle on vertices $0$, $u$, and $v$. Then we have
\[\dist(u, v)^2 = ||u||^2 + ||v||^2 - 2||u||||v|| \cos\theta\]
where $\theta$ is the angle between $||u||$ and $||v||$. Now as $||u||$ and $||v||$ are both between $1$ and $(1 + \epsilon)$ and the entire segment from $u$ to $v$ lives outside of the unit sphere we have that the altitude from the origin to the segment from $u$ to $v$ has length at least 1. Letting $\theta_1$ denote the angle between $u$ and the altitude and $\theta_2$ denote the angle between $v$ and the altitude as in Figure \ref{fig:circle}, we have
\[\cos \theta_1 \geq \frac{1}{1 + \epsilon},\]
and
\[\cos \theta_2 \geq \frac{1}{1 + \epsilon}.\]
 From this and $0 \leq \theta \leq \pi$, it follows that 
\[\cos \frac{\theta}{2} \geq \frac{1}{1 + \epsilon}, \]
and so 
\begin{eqnarray*}
\dist(u, v)^2 &\leq& 2(1 + \epsilon)^2 - 2 \left(2\frac{1}{(1 + \epsilon)^2} - 1\right) \\
&=& 4 + 4\epsilon + 2 \epsilon^2 - \frac{4}{(1 + \epsilon)^2}  \\
&=& \frac{8 \epsilon + 4 \epsilon^2}{(1 + \epsilon)^2} + 4 \epsilon + 2 \epsilon^2 \\
&<& 12 \epsilon + 6 \epsilon^2.
\end{eqnarray*}
\end{proof}

\begin{figure}[h]
\includegraphics[width = 2 in]{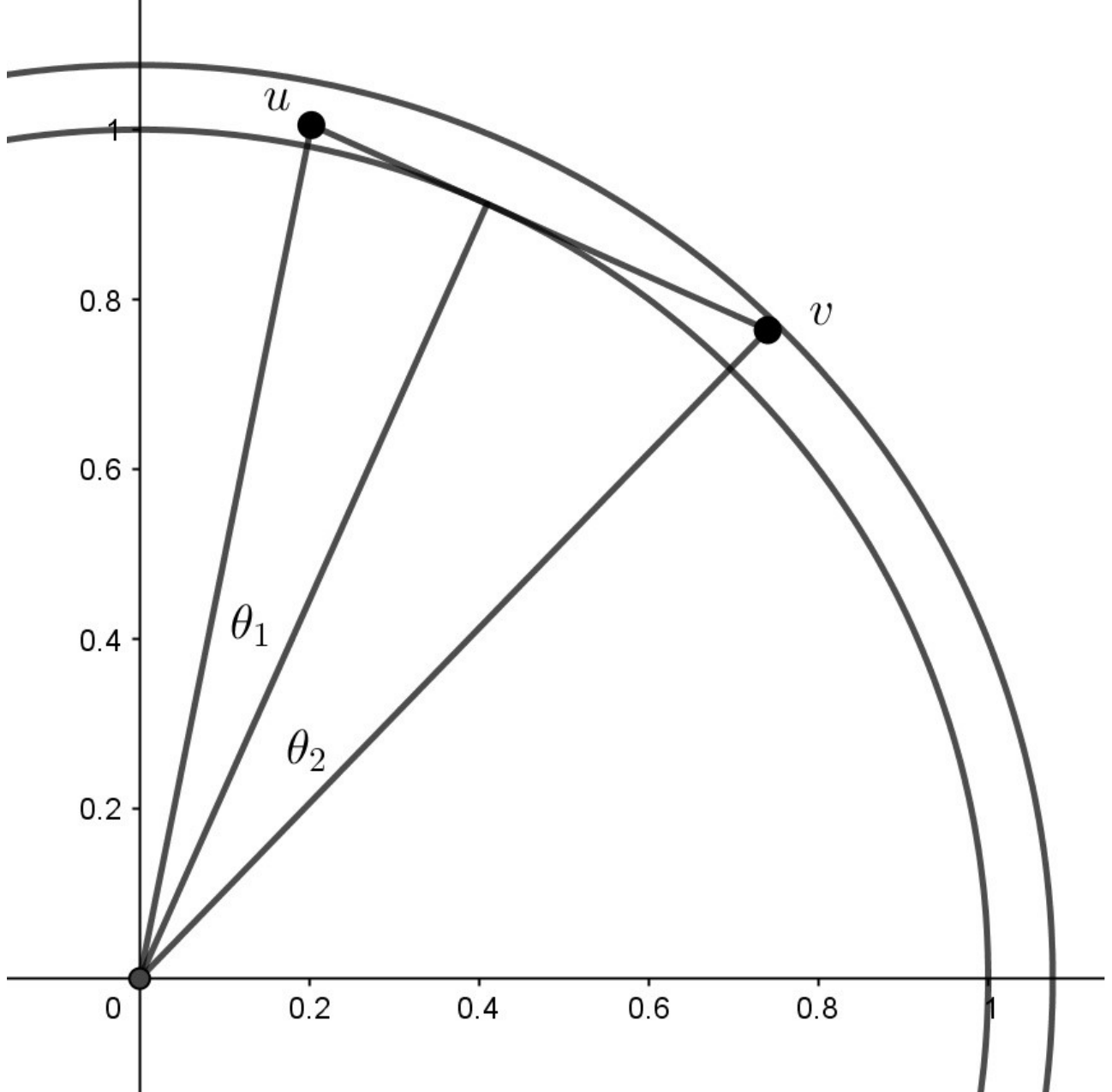}
\caption{Figure drawn using Geogebra \cite{geogebra} for proof of Claim \ref{triangle}}\label{fig:circle}
\end{figure}
We point out that what we've proved here implies that under the assumptions of Claim \ref{triangle} every edge of $P$ has length at most $\sqrt{12 \epsilon + 6 \epsilon^2}$, but the claim is quite a bit stronger; this will be convenient later.
\begin{claim}\label{shortedges}
For $\epsilon > 0$ bounded above by a sufficiently small constant and $P \sim P_1(d, m)$ with probability at least \[1 - \binom{m}{d} \left(1 - \frac{\sqrt{ \epsilon}^{d - 1}v_{d - 1}}{2dv_d} \right)^{m - d}\] every facet of $P^{\circ}$ has diameter at most $\sqrt{12 \epsilon + 6 \epsilon^2}$. Thus with high probability every facet of $P^{\circ}$ has diameter $O((\log m/m)^{1/(d - 1)})$ where the implied constant depends only on $d$. 
\end{claim}
\begin{proof}
This is immediate from Lemma \ref{HausdorffOutside} and Claim \ref{triangle}.
\end{proof}
Now that we have that all of the edges of $P^{\circ}$ will be short we turn our attention to the number of vertices above new facets in $Q$. This argument will be similar to the argument from Lemma \ref{HausdorffOutside}.
\begin{claim}\label{DiscreteCaps}
If $P$ is a nondegenerate simple $d$-polytope on at most $n$ vertices and $Q \sim \B(P, p)$ with $p \in [0, 1]$ then for any $t$ the probability that $Q$ has a new facet with at least $t$ vertices of $P$ above it is at most 
\[\binom{n}{d} p^d (1 - p)^t.\]
\end{claim}
\begin{proof}
Let $\sigma$ be a facet determined by $d$ vertices of $P$ which are not on any facet of $P$ with $t$ vertices above the hyperplane of $\sigma$. Then $\sigma$ is a facet of $Q$ only if every vertex of $\sigma$ belongs to $Q$ and the vertices above it do not belong to $Q$. Thus the probability that $\sigma$ is a facet of $Q$ is therefore at most $p^d (1 - p)^t$. As there are at most $\binom{n}{d}$ choices for $\sigma$ we have that the expected number of new facets with exactly $t$ vertices above each of them is at most
\[\binom{n}{d} p^d (1 - p)^t.\]
The claim now follows by Markov's inequality.
\end{proof}

Now we have all the pieces in place to bound from below how close to the origin the facets of $Q$ may be.
\begin{lemma}\label{insideprob}
If $d$ is fixed and $p$ and $m$ are such that $pm^{1/(d - 1)}$ tends to infinity as $m$ tends to infinity then with high probability every facet of $Q \sim P_2(d, m, p)$ is at distance at least
\[1 - \Omega\left(\frac{ \log^2(pm) \log ^{2/(d - 1)}m}{p^2m^{2/(d - 1)}}\right) \]
from the origin.
\end{lemma}
\begin{proof}
By Theorem \ref{vertexconcentration} there exists a constant $c_1$ so that with high probability $P^{\circ}$ for $P \sim P_1(d, m)$ has at most $c_1 m$ vertices. In this case we may apply Claim \ref{DiscreteCaps} with $n = c_1 m$ and we have that there is $c_2 > 1$ so that with high probability every new facet of $Q \sim \B(P^{\circ}, p)$ has at most  $c_2 (\log (pm))/p$ vertices above it.  Moreover by Claim \ref{shortedges} there exists $c_3$ so that with high probability every facet of $P^{\circ}$ has diameter at most $c_3((\log m/m)^{1/(d - 1)})$. In the likely case that all of these conditions hold we have that for any hyperplane $\sigma$ determining a new facet of $Q$ there is a path of length at most 
\[\frac{c_2c_3 \log(pm) \log^{1/(d - 1)}m}{pm^{1/(d - 1)}}\]

from the point of $P^{\circ}$ where the orthogonal vector to $\sigma$ crosses $P^{\circ}$ to a vertex of $P^{\circ}$ on $\sigma$. This holds because this point of intersection, $w$ belongs to some facet of $P^{\circ}$ so it is within distance at most $c_3((\log m/m))^{1/(d - 1)}$ of some vertex of $P^{\circ}$ and from here we may follow a path in the graph induced by $\mathcal{C}(\sigma)$ back to some vertex of $\sigma$. As we have bounded the number of vertices above $\sigma$, the number of edges of this path is at most $c_2 (\log (pm))/p  - 1$ and each edge has length at most $c_3((\log m/m))^{1/(d - 1)}$ by the diameter assumption. However as every vertex of $\sigma$ and $w$ are all at distance at least 1 from the origin, the distance from the origin to $\sigma$ is at least 
\[1 - \frac{1}{2}\left(\frac{c_2c_3 \log(pm) \log^{1/(d - 1)}m}{pm^{1/(d - 1)}}\right)^2.\]
To see this, consider the triangle determined by $w$, the origin, and $v$ an arbitrary vertex of $P^{\circ}$ on $\sigma$. We have that $\dist(w, v) \leq \frac{c_2c_3 \log(pm) (\log m)^{1/(d - 1)}}{pm^{1/(d - 1)}}$, $||w|| \geq 1$, and $||v|| \geq 1$. Now the distance from the origin to the hyperplane $\sigma$ is given by the length of the projection of $v$ onto $w$, as $w$ is a normal vector to $\sigma$. Using $\langle w, v \rangle$ to denote the usual inner product of $w$ and $v$ we have:
\begin{eqnarray*}
\dist(0, \sigma) &=& \frac{|\langle w, v \rangle|}{||w||} \\
&\geq& \frac{||w||^2 + ||v||^2 - \dist(w, v)^2}{2||w||} .
\end{eqnarray*}
Note that equality holds on the second line as long as $\dist(w, v)^2 < ||w||^2 + ||v||^2$, otherwise its a trivial, negative lower bound.

Now as both $||w||$ and $||v||$ are at least one we have that $||w||^2 + ||v||^2 \geq 2||w||$ and so
\[\dist(0, \sigma) \geq 1 - \frac{\dist(w, v)^2}{2}.\]
\end{proof}

\begin{proof}[Proof of Theorem \ref{Hausdorffdistance}]
There exists a constant $\gamma$ depending on $d$ so that with high probability the boundary of $Q \sim P_2(d, m, p)$ lives outside the sphere of radius 
\[1 - \frac{ \gamma \log^2(mp) \log^{2/(d-1)} m}{p^2m^{2/(d - 1)}}\]
centered at the origin by Lemma \ref{insideprob}. Moreover there exists $\Gamma$ depending on $d$ so that with high probability $Q$ is contained entirely in the sphere of radius
\[1+ \frac{\Gamma \log^{2/(d - 1)} (pm^d)}{m^{2/(d - 1)}}  \]
centered at the origin by Lemma \ref{HausdorffOutside}.
\end{proof}
As a corollary to the proof of Theorem \ref{Hausdorffdistance} we have the following result about the volume of $Q \sim P_2(d, m, p)$.

\begin{corollary} 
For $d$ and $p$ fixed, with high probability the volume of $Q \sim P_2(d, m, p)$ is between 
\[v_d \left(1 - \frac{ \gamma \log^2(mp) \log^{2/(d - 1)}m}{p^2m^{2/(d - 1)}}\right)^{d}\]
and 
\[v_d \left(1+\frac{\Gamma\log^{2/(d-1)}(pm^d)}{m^{2/(d -1)}} \right)^d\]
where $\gamma$ and $\Gamma$ are constants depending only on $d$. 
\end{corollary}

\section{Concluding remarks}
Figure \ref{fig:examples} shows results of experiments examining $P_2(d, m, p)$ conducted in \polymake \cite{polymake:2000}. For these experiments we consider a stochastic process version of $\B(P, p)$. In each experiment we begin with $P^{\circ}$ for $P \sim P_1(d, m)$. Letting $n$ denote the number of vertices of $P^{\circ}$ the experiment is conducted by starting with $P^{\circ}$ and at each step successively deleting a random subset of vertices of the current polytope size $\lfloor 0.5 n \rfloor$ until fewer than $d + 1$ vertices remain. 

We ran this experiment 10 times for $d = 3$, $4$, and $5$ and $m$ selected for each $d$ so that the starting polytope $P^{\circ}$ had about 500 vertices. For $d = 3$, $m = 252$ (in this case there are always exactly 500 vertices since we have a simple 3-polytope with 252 facets), for $d = 4$, $m = 89$, and for $d = 5$, $m = 37$. The results in Figure \ref{fig:examples} show for each $q = 1 - p = 0, 0.05, 0.1, ..., 0.95, 1$, the mean and standard error for $f_{d - 1}/n$ across 10 trials for each value of $d$. Recall that $n$ is the number of vertices at the start of the process. For the random 3 polytopes in all trials $n = 500$. For the random 4-polytopes $n$ was on average 504.7 with a standard error of 6.68, and for the random 5-polytopes $n$ was on average $489.6$ with a standard error of 14.10.

For readability we have interpolated in a piecewise linear way in between the values of $q$ where we checked the number of facets. Moreover we point out that due to rounding since we delete $\lfloor 0.5 n \rfloor$ vertices at each step the plot is slightly off, for example the first bar at $q = 0.5$ reflects the average behavior across the experiments when the first $\lfloor 0.5 n \rfloor$ vertices are deleted, not literally when exactly five percent of the vertices are deleted. In any case, the results of these \polymake experiments alone suggest that it would be interesting to understand the behavior of $P_2(d, m, p)$ as $p$ ranges from $0$ to $1$, for instance to establish $p_0 \in [0, 1]$ which maximizes the asymptotic complexity.

\begin{figure}[h]
\centering
\includegraphics[scale = 1]{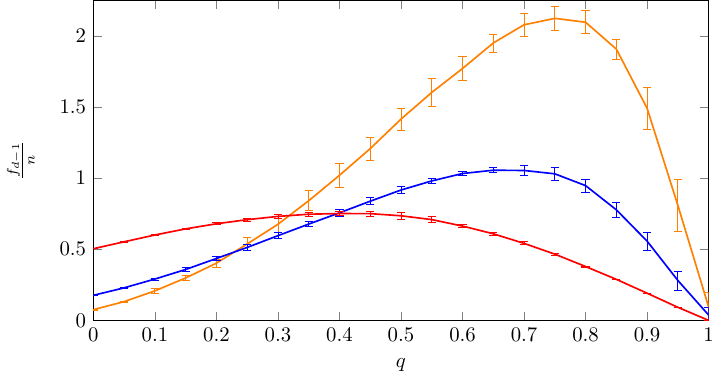}
\caption{Density of deleted vertices $q$ against the normalized number of facets for a random 3-polytope (red), 4-polytope (blue), and 5-polytope (orange)}\label{fig:examples}.
\end{figure}

 To this point we can only say something about the asymptotic complexity for small values of $q$. However, the restriction to small values of $q$ in the proof of Theorem \ref{upperbound} likely has more to do with the method of proof than anything about the random model; it seems completely reasonable to expect that linear complexity will hold for all fixed values of $q \in [0, 1]$. Essentially the argument to prove Theorem \ref{upperbounddeterministic} studies the combinatorial structure of $Q \sim \B(P, p)$ by studying the graph structure of the graph induced on the vertices of $P$ not included in the random sample. This approach is also implicit in the proof of Lemma \ref{insideprob}, a key lemma for the proof of Theorem \ref{Hausdorffdistance}.

Underlying the proof of Theorem \ref{upperbounddeterministic} is a question about percolation for a random induced subgraph of the 1-skeleton of $P$ where each vertex is sampled independently with probability $q$. If $q$ is small, something implicit in our theorem is that every component of the resulting graph will be small, of order $O(\log n)$. Connected subgraphs of $Q$ in turn tell us something about the number of facets of $P$. As $q$ grows larger there should come a point where the induced graph contains a \emph{giant component} as occurs in the Erd\H{o}s--R\'{e}nyi model and so the argument would need to be changed.

\section*{Acknowledgments}
The author is grateful to Michael Joswig and to Matthias Reitzner for helpful comments on an earlier draft. The author also gratefully acknowledges funding by Deutsche Forschungsgemeinshaft (DFG, German Research Foundation) Graduiertenkolleg ``Facets of Complexity" (GRK 2434).

\section*{Appendix: An application of the doubly-random model}

This appendix is motivated by an application of random polytopes considered in \cite{JKR}.  In \cite{JKR}, Joswig, Kaluba, and Ruff consider randomized methods for approximating the convex hull of a finite set for outlier detection in machine learning. Given a finite set $X$, arising from a real data set, the authors  use a \emph{dual bounding body} construction to approximate $X$; the goal is then to use the dual bounding body to decide whether a new point is close to $X$. This dual bounding body construction of \cite{JKR} is based on the dual $P^{\circ}$ of a random polytope $P \sim P_1(d, m)$, and an important step of their procedure is to determine the center of the dual bounding body generated; that is to find the average position of its vertices. Unfortunately, the dual bounding body may have very many vertices and so precisely determining the center may not be computationally feasible. To this end, the question becomes how well can a sample of the vertices from the dual bounding body approximate the center. Because the dual bounding body is a modified version of the dual of $P \sim P_1(d, m)$ taking a sample of its vertices is closely tied to $Q \sim P_2(d, m, p)$. For the purposes of applying what's been established here to the dual bounding body construction of \cite{JKR} we are interested in the following question.

\begin{question}\label{approximationquestion}
For $m > d \geq 2$, $\pi \in [0, 1]$ and $\epsilon > 0$, how large should be $N = N(m, d, \pi, \epsilon)$ so that with probability at least $1 - \pi$ the center of a random sample of $N$ vertices from $P^{\circ}$, $P \sim P_1(d, m)$ is within distance $\epsilon$ of the origin?
\end{question}

Here we do have to be a bit careful about the random sample is selected.  The vertices of $P_2(d, m, p)$ are sampled independently with probability $p$, thus the probability that a single vertex is included is the same for all vertices of the initial random polytope. In practice though to sample in this way one would need to know all the vertices of $P^{\circ}$. Instead one may take a sample of $N$ vertices of $P^{\circ}$ via optimizing uniform random linear objective functions on the unit sphere. This does not assign equal probability to the vertices of $P^{\circ}$ but for $N$ fixed and $m \rightarrow \infty$, in practice this likely does not make much of a difference so we ignore the technicalities between these two ways of randomly choosing a set of vertices.

One key step toward answering Question \ref{approximationquestion} is the result about the Hausdorff distance of $Q \sim P_2(d, m, p)$ to the unit sphere, especially Lemma \ref{HausdorffOutside}. A second step is the following lemma about how close the average value of points on a sphere of radius $R$ in $\R^d$ is to the origin.
\begin{lemma}\label{spherecenter}
Fix $d$ a positive integer and $R, \pi, \epsilon > 0$ then for $N \geq \frac{2R^2}{\epsilon^2}\left(1 + \frac{2}{d} \log \left(\frac{1}{\pi}\right) \right)$ points $v_1, ..., v_N$ chosen uniformly at random from the boundary of the $d$-dimensional ball of radius $R$ centered at the origin in $\R^d$ one has that with probability at least $1 - \pi$, 
\[\frac{1}{N}||v_1 + \cdots + v_N|| \leq \epsilon.\]
\end{lemma}
In the case that $d = 1$ this lemma gives a concentration inequality for summing up $N$ random numbers each of which is uniformly selected from $\{-1, 1\}$. This is a problem that has been considered in the past, see for example Appendix A of \cite{AS}, but it seems that a result like Lemma \ref{spherecenter} as precise as we need here and for $d > 1$ does not appear in the literature. 
\begin{proof}[Proof of Lemma \ref{spherecenter}]
Let $d$, $R$, $\pi$, and $\epsilon$ be given as in the statement. Let $X_i$ denote the random variable given by projection of a random point on the sphere of radius $R$ centered at the origin in $\R^d$ onto the $i$th coordinate. By symmetry $\E(X_i) = 0$ and for any pair $1 \leq i \leq j \leq n$, $X_i$ and $X_j$ have the same distribution. Moreover we always have
\[\sum_{i = 1}^d (X_i)^2 = R.\]
Thus $\E(X_i^2) = R^2/d$, hence $\Var(X_i) = \E(X_i^2) - [E(X_i)]^2 = R^2/d$. If we sample $v_1, ..., v_N$, uniformly at random from the sphere of radius $R$ in $\R^d$, then by the central limit theorem the average of the $i$th coordinate of $v_1, ..., v_N$ will be asymptotically in $N$ distributed as a Gaussian with mean $0$ and variance $R^2/(dN)$.

Now if the coordinate averages were independent from one another we could approximate 
\[\frac{||v_1 + \cdots + v_{N}||^2}{N^2}\]
 by the sum of $d$ independent copies of $[\mathcal{N}(0, R^2/(dN))]^2$. Such a sum is distributed as 
\[\frac{R^2}{dN} \chi^2(d) \]
where $\chi^2(d)$ denotes the chi-squared distribution with $d$ degrees of freedom.

The coordinate averages of course are not independent from one another, however it does hold that
\[\Pr\left(\frac{||v_1 + \cdots + v_N||^2}{N^2} \geq \epsilon^2\right) \leq \Pr\left(\frac{R^2}{dN}\chi^2(d) \geq \epsilon^2\right). \]
That is, sampling each coordinate average independently from a normal distribution with mean zero and variance $R^2/(d N)$ could only increase the probability that the square of the sum of the averages exceeds $\epsilon^2N^2$. Next by Chernoff bound we have that
\[\Pr \left(\frac{R^2}{dN}\chi^2(d) \geq \epsilon^2\right) \leq \left(\frac{\epsilon^2 N}{R^2} \exp(1 - \epsilon^2 N/R^2) \right)^{d/2}.\]
Now we see that by setting 
\[N \geq \frac{2R^2}{\epsilon^2}\left(1 + \frac{2}{d} \log \left(\frac{1}{\pi}\right) \right) \]
we have that the right hand side is at most $\pi$. Indeed

\begin{eqnarray*}
&&\left[2 \left(1 - \frac{2}{d} \log \pi \right) \exp\left(1 - \left(2\left(1 + \frac{2}{d} \log\left(\frac{1}{\pi}\right) \right)\right)\right)\right]^{d/2}\\
&& \leq 2^{d/2} \pi^{-1} \exp \left(-d/2 + 2\log(\pi)\right) \\
&& \leq \pi.
\end{eqnarray*}
\end{proof}

We close by describing how one could use Lemma \ref{HausdorffOutside} and Lemma \ref{spherecenter} to answer Question \ref{approximationquestion}. As we are already ignoring the discrepancy in two methods of choosing the vertices anyway, we omit some details here and don't make a precise statement, but sketch the idea.

Given $d$, $m$, $\epsilon$ and $\pi$, one can apply Lemma \ref{HausdorffOutside} to find $p$ so that every vertex of $Q \sim P_2(d, m, p)$ is within distance $1 + \epsilon$ of the origin. This $p$ correspond to $N_1 = pF(d)m$. Next we take $R$ of Lemma \ref{spherecenter} to be $1 + \epsilon$ and we take $N_2$ as in Lemma \ref{spherecenter}, if $N_2 < N_1$ then $N = N_2$ gives an answer to Question \ref{approximationquestion} for some easily computable scalar multiples of $\epsilon$ and $\pi$.  Some values of $\epsilon$ and $\pi$ are too small relative to $m$ and $d$ for $N_2$ to be smaller than $N_1$, and in this case we wouldn't get an answer, but this is less of a concern as $m$ grows.

\bibliography{ResearchBibliography}
\bibliographystyle{amsplain}
\end{document}